\documentclass[11pt,a4paper,reqno]{amsart}
\usepackage{xcolor}
\usepackage{amsfonts}
\usepackage{amsmath}
\usepackage{amssymb}
\usepackage{hyperref}

\usepackage{subfigure}

\usepackage{graphicx}
\usepackage{epstopdf}

\usepackage{amsthm}

\usepackage{mathrsfs}
\usepackage{dutchcal}
\usepackage{tikz}

\numberwithin{equation}{section}
\newtheorem{theorem}{Theorem}[section]
\newtheorem{lemma}[theorem]{Lemma}
\newtheorem{remark}[theorem]{Remark}
\newtheorem{corollary}[theorem]{Corollary}
\newtheorem{proposition}[theorem]{Proposition}
\newtheorem{definition}[theorem]{Definition}
\newtheorem{example}[theorem]{Example}
\def\eqref#1{(\ref{#1})}
\allowdisplaybreaks

\def\enddoc{\end{document}}

\def\FRAME#1#2#3#4#5#6#7#8
{
\begin{figure}[ptbh]
	\begin{center}
		\includegraphics[height=#3]{#7}
		\caption{#5}
		\label{#6}
	\end{center}
\end{figure}
}

\begin{document}

\author{Lu Hao}
\address{Universit\"{a}t Bielefeld, Fakult\"{a}t f\"{u}r Mathematik, Postfach 100131, D-33501, Bielefeld, Germany}
\email{lhao@math.uni-bielefeld.de}
	
\author{Yuhua Sun}

\address{School of Mathematical Sciences and LPMC, Nankai University, 300071
		Tianjin, P. R. China}
\email{sunyuhua@nankai.edu.cn}
\title[$L^p$-Parabolicity and Nonlinear Elliptic Systems]{$L^p$-Parabolicity, the $L^q$-Liouville Property, and Nonlinear Elliptic Systems on Weighted Graphs}

\thanks{\noindent L. Hao was funded by the Deutsche Forschungsgemeinschaft
(DFG, German Research Foundation), Project-ID 317210226, SFB 1283. Y. Sun
was funded by the National Natural Science Foundation of China (Grant
No. 12371206).}


\begin{abstract}
We establish an equivalence between $L^p$-parabolicity, the $L^q$-Liouville property for nonnegative superharmonic functions, and the rigidity of nonnegative solutions to the elliptic system
\[
\left\{
\begin{array}{rcl}
-\Delta u\geq0,\\
\Delta\bigl(|\Delta u|^{p-2}\Delta u\bigr)\geq0,
\end{array}
\right.
\]
where $1<p<\infty$ and $q=p/(p-1)$.  More precisely, the three properties above are also equivalent to divergence of the $L^q$-Green kernel.  For $1<p\leq2$, we obtain a volume-growth criterion under only a uniform positive lower bound on the vertex measure.  Under $\mathrm{(VD)}$, $\mathrm{(PI)}$, and $(P_0)$, an auxiliary weighted graph whose associated random walk is aperiodic yields a two-sided Green-function estimate, an exact series criterion for $L^p$-parabolicity, and a critical logarithmic sufficient condition.  Cayley graphs of polynomial growth illustrate the sharpness of the resulting threshold.
\end{abstract}

\maketitle

\section{\texorpdfstring{Introduction}{Introduction}}

The notion of parabolicity arose in the classification of Riemannian surfaces \cite{Ahlfors}. On manifolds and weighted graphs, it admits several equivalent formulations from stochastic analysis, elliptic PDE, and potential theory. These include recurrence of the associated diffusion or random walk, nonexistence of a positive Green function, Liouville rigidity for nonnegative superharmonic functions, and vanishing harmonic capacity of compact or finite sets. These equivalences link the large-scale geometry of the underlying space to the qualitative behavior of harmonic and superharmonic functions.

Liouville-type theorems are a central expression of this rigidity. The classical Liouville theorem states that every bounded harmonic function on $\mathbb R^n$ is constant, while any harmonic function in $L^q(\mathbb R^n)$, $1<q<\infty$, must vanish. For nonnegative superharmonic functions, the corresponding $L^q$-Liouville property on $\mathbb{R}^n$ with $n\geq3$ holds precisely in the range $1<q\leq\frac{n}{n-2}$. The exponent $\frac{n}{n-2}$ is sharp, as illustrated by
\[
u(x)=\frac{1}{(1+|x|^2)^{\frac{1}{\lambda-1}}},
\]
where $\lambda>\frac{n}{n-2}$. This function is positive and superharmonic and belongs to $L^q(\mathbb{R}^n)$ whenever $q>\frac{n(\lambda-1)}{2}$, see \cite{Miti}.

The Euclidean threshold motivates the $L^q$-Liouville property: every nonnegative superharmonic function in $L^q$ is constant (and hence vanishes when the underlying measure is infinite), with the usual constancy formulation when $q=\infty$. On Riemannian manifolds, Grigor'yan, Pessoa, and Setti \cite{GPS} related this property to $L^p$-parabolicity defined through an $L^p$-capacity constructed from the Green kernel, where $p$ and $q$ are H\"older conjugate exponents, and derived volume-growth criteria by combining potential theory with Green-function estimates. Their general criterion has the scale $V(o,r)\lesssim r^{2p}/\log r$ for $1<p\leq2$, while their sharper all-$p$ conclusions use nonnegative Ricci curvature or model-manifold geometry. Their work is therefore the closest manifold counterpart to the present paper.

In this paper we use the graph Green operator to define an $L^p$-capacity and the corresponding notion of $L^p$-parabolicity. The manifold theory in \cite{GPS} provides the conceptual starting point, but passing to graphs introduces periodicity and requires killed-walk and exhaustion arguments. We therefore develop the potential theory directly on weighted graphs and derive additional characterizations and volume-growth consequences. This $L^p$-capacity is distinct from the $p$-capacity studied in \cite{AFS}, which is defined through the $p$-energy in \eqref{pcap}.

The difference can be seen from several complementary viewpoints. From the probabilistic viewpoint, the Green kernel is generated by a discrete-time reversible random walk, and a loop-free or bipartite graph may be periodic; hence a lower heat-kernel estimate cannot in general hold at every integer time. From the potential-theoretic viewpoint, the capacity and Liouville arguments must be built from killed random walks, Green operators on subsets, and exhaustion by finite subgraphs. From the elliptic viewpoint, the rigidity of the higher-order nonlinear system \eqref{equ} has to be proved without the smooth differential calculus available on manifolds. Consequently, the graph case requires direct Green-operator arguments and an explicit treatment of periodicity rather than a theorem-by-theorem transcription of the manifold proof.

To formulate these distinctions precisely, we now fix the graph setting and the basic operators. Throughout the paper, let $G=(V,E)$ be an infinite, connected, locally finite graph, where $V$ and $E$ denote the vertex and edge sets, respectively. We write $x\sim y$ when $x$ and $y$ are joined by an edge. An edge of the form $x\sim x$ is called a loop, and $G$ is called simple when it has no loops.

Let $\mu:V\times V\to[0,\infty)$ be a symmetric edge weight, written $\mu_{xy}=\mu(x,y)$, such that $\mu_{xy}>0$ if and only if $x\sim y$. The vertex measure is
\[
\mu(x):=\sum_{y\sim x}\mu_{xy}.
\]
We refer to $(V,E,\mu)$, or simply $(V,\mu)$, as a weighted graph.

For $x,y\in V$, let $d(x,y)$ be the length of a shortest path joining $x$ to $y$. Fix a vertex $o\in V$. For $r\geq0$, set
\[
B(o,r):=\{x\in V:d(o,x)\leq r\},
\qquad
V(o,r):=\mu(B(o,r)).
\]

The weighted graph determines a reversible random walk $\{X_n\}_{n\geq0}$ with transition kernel
\begin{equation}\label{P-tran}
P(x,y)=
\begin{cases}
\dfrac{\mu_{xy}}{\mu(x)}, & x\sim y,\\[4pt]
0, & \text{otherwise}.
\end{cases}
\end{equation}
It follows that $\mu(x)P(x,y)=\mu(y)P(y,x)$. Conversely, a reversible Markov chain  determines a weighted graph by setting $\mu_{xy}=\mu(x)P(x,y)$ and declaring $x\sim y$ whenever $\mu_{xy}>0$. We write
\[
P_n(x,y):=\mathbb P_x(X_n=y),
\]
so that $P_0(x,y)=\delta_x(y)$ and $P_1(x,y)=P(x,y)$.

The random walk, and hence the graph, is called recurrent if it returns
to its starting vertex almost surely; otherwise it is called transient.
Classical parabolicity of a graph is equivalent to recurrence.

For $U\subset V$, let $\ell(U)$ be the space of real-valued functions on $U$, let $\ell_0(U)$ be its subspace of finitely supported functions, and let $\ell^+(U)$ be the cone of nonnegative functions. For $1\leq p<\infty$, set
\[
L^p(U):=\left\{f\in\ell(U):
\sum_{x\in U}|f(x)|^p\mu(x)<\infty\right\},
\qquad
L^p_+(U):=L^p(U)\cap\ell^+(U).
\]

The normalized graph Laplacian is
\begin{align*}
\Delta u(x)
&=\frac1{\mu(x)}\sum_{y\sim x}\mu_{xy}\bigl(u(y)-u(x)\bigr)\\
&=\sum_{y\in V}P(x,y)\bigl(u(y)-u(x)\bigr)
=(P-I)u(x),
\end{align*}

where
\[
Pu(x):=\sum_{y\in V}P(x,y)u(y).
\]

A function $u\in\ell(V)$ is harmonic if $\Delta u=0$ and
superharmonic if $-\Delta u\geq0$ on $V$.

The heat kernel with respect to $\mu$ is $p_n(x,y):=P_n(x,y)/\mu(y)$. Reversibility gives $p_n(x,y)=p_n(y,x)$. The Green function is
\begin{equation}\label{gf}
g(x,y):=\sum_{n=0}^{\infty}p_n(x,y),
\end{equation}
which is allowed to take the value $+\infty$. In particular,
$g(x,y)=g(y,x)$. For $u\in\ell^+(V)$, define the Green operator by
\[
Gu(x):=\sum_{y\in V}g(x,y)u(y)\mu(y).
\]

For a nonnegative function $\nu$ on $V$ and a set $A\subset V$, write $\nu(A):=\sum_{x\in A}\nu(x)$. Fix $1\leq p<\infty$, and let $q$ be its H\"older conjugate exponent. For a finite set $K\subset V$, define
\begin{align}\label{CpK}
C_p(K):=
\sup\left\{\nu(K)^p:
\nu=f\mu,\ f\in\ell^+(K),\ \|Gf\|_{L^q(V)}\leq1\right\}.
\end{align}
where
\[
\|Gf\|_{L^q(V)}
:=\left(\sum_{x\in V}|Gf(x)|^q\mu(x)\right)^{1/q},
\]
while for $q=\infty$, $\|Gf\|_{L^\infty(V)}$ is the usual supremum norm.

\begin{definition}\label{def-lp-para}\rm

A weighted graph $(V,\mu)$ is called $L^p$-parabolic if $C_p(K)=0$ for every finite set $K\subset V$.
\end{definition}

\begin{definition}\label{def-lq-liouv}\rm

For $1\leq q\leq\infty$, a weighted graph $(V,\mu)$ is said to admit the $L^q$-Liouville property if every nonnegative superharmonic function in $L^q(V)$ is constant. 
\end{definition}

For comparison, we recall another notion of $p$-parabolicity that should not be confused with $L^p$-parabolicity. A graph is called $p$-parabolic if the following $p$-capacity of every finite set vanishes:
\begin{align}\label{pcap}
\operatorname{Cap}_p(K):=
\inf\left\{\sum_{x,y\in V}\mu_{xy}|u(y)-u(x)|^p:
 u\in\ell_0(V),\ u\geq1\text{ on }K\right\},
\end{align}

The variational notion has a separate literature. Holopainen \cite{Hol1}
developed nonlinear potential theory on Riemannian manifolds, while
Troyanov \cite{Tro} gave a broader account of manifold parabolicity. For
networks and graphs, Yamasaki \cite{Yam} treated parabolic and hyperbolic
networks, Soardi \cite{Sor} related parabolicity to polynomial growth, and
Saloff-Coste \cite{S95} studied inequalities for $p$-superharmonic
functions. Coulhon, Holopainen, and Saloff-Coste \cite{CHS} connected
Harnack inequalities with the corresponding nonlinear hyperbolicity
theory. These works include capacity and Liouville characterizations in
their respective settings; see \cite{Hol1,S95}. Adriani, Fischer, and
Setti \cite{AFS} recently developed such characterizations for
$p$-energy on locally summable graphs.
The same variational theory also supplies volume-growth criteria; see
\cite{CHS,S95}. Under the completeness and local hypotheses imposed in
those sources, a representative sufficient condition is
\[
V(o,r)\lesssim r^p(\log r)^{p-1}
\]
for all sufficiently large $r$.

The nonlinear system considered here is motivated by the related
higher-order notion of biparabolicity. A manifold $M$ is called
biparabolic if every nonnegative solution of
\begin{equation}\label{bipeq}
\left\{
\begin{array}{rcl}
-\Delta u\geq0,\\
\Delta^2u\geq0
\end{array}
\right.
\end{equation}
on $M$ is harmonic. Faraji and Grigor'yan \cite{FG} proved the following
volume criterion: if $M$ is geodesically complete and
\[
V(x_0,r)\lesssim \frac{r^4}{\log r}
\]
for all sufficiently large $r$, then $M$ is biparabolic.

Motivated by this fourth-order case, for general $1<p<\infty$ we consider
the nonlinear elliptic system
\begin{equation}\label{equ}
\left\{
\begin{array}{rcl}
-\Delta u\geq0,\\
\Delta\bigl(|\Delta u|^{p-2}\Delta u\bigr)\geq0.
\end{array}
\right.
\end{equation}
When $p=2$, \eqref{equ} reduces to \eqref{bipeq}.

\begin{definition}\label{def-bipar}\rm

A weighted graph $(V,\mu)$ is called biparabolic if every nonnegative solution of \eqref{bipeq} is harmonic.
\end{definition}

For $1<p<\infty$ and $q=p/(p-1)$, define the $L^q$-Green kernel by
\[
g_q(x,y):=\sum_{z\in V}g(x,z)g(z,y)^{q-1}\mu(z).
\]

Our first main result is the following equivalence.
\begin{theorem}\label{tm1}\rm
Let $1<p<\infty$ and $q=p/(p-1)$. Then the following conditions are equivalent:
\begin{enumerate}
\item[(a)] $(V,\mu)$ is $L^p$-parabolic;
\item[(b)] $(V,\mu)$ has the $L^q$-Liouville property;
\item[(c)] Every nonnegative solution of \eqref{equ} is harmonic;
\item[(d)] $g_q(x,y)=\infty$ for some, equivalently every,
$x,y\in V$.
\end{enumerate}
\end{theorem}

\begin{remark}\rm
This Green kernel characterization connects the qualitative equivalence in Theorem \ref{tm1} to the volume-growth criteria below. In particular, for $p=2$, Theorem \ref{tm1} identifies $L^2$-parabolicity with biparabolicity.
\end{remark}

We next translate the Green kernel characterization into volume-growth conditions. Before stating the results, we introduce volume doubling
\hyperref[VD]{$\mathrm{(VD)}$}, the Poincar\'e inequality \hyperref[PI]{$\mathrm{(PI)}$}, and the controlled-weight condition \hyperref[p0]{$(P_0)$}.
\begin{definition}\rm
	
	The weighted graph $(V,\mu)$ satisfies condition \hyperref[VD]{$\mathrm{(VD)}$} if
	\phantomsection\label{VD}
	\[
	V(x,2r)\lesssim V(x,r)
	\tag{$\mathrm{VD}$}
	\]
	for every $x\in V$ and $r>0$.
\end{definition}

\begin{definition}\rm
	
	The weighted graph $(V,\mu)$ supports  \hyperref[PI]{$\mathrm{(PI)}$} if
	\phantomsection\label{PI}
	\[
	\sum_{x\in B(x_0,r)}|f(x)-f_B|^2\mu(x)
	\lesssim r^2
	\sum_{x,y\in B(x_0,2r)}\mu_{xy}|f(y)-f(x)|^2
	\tag{$\mathrm{PI}$}
	\]
	for every $x_0\in V$, $r>0$, and $f\in\ell(V)$, where
	\[
	f_B:=\frac1{V(x_0,r)}\sum_{x\in B(x_0,r)}f(x)\mu(x).
	\]
\end{definition}

\begin{definition}\rm
	
	The weighted graph $(V,\mu)$ satisfies condition \hyperref[p0]{$(P_0)$} if there exists $\alpha>0$ such that
	\phantomsection\label{p0}
	\[
	\frac{\mu_{xy}}{\mu(x)}\geq\alpha
	\qquad\text{whenever }y\sim x.
	\tag{$P_0$}
	\]
\end{definition}

\begin{theorem}\label{vol-lp}\rm

Let $1<p\leq2$, and assume that there is a constant $b>0$ such that $\mu(x)\geq b$ for every $x\in V$. If, for some $o\in V$,
\begin{equation}\label{vol-growth-intro}
V(o,r)\lesssim \frac{r^{2p}}{\log r}
\end{equation}
for all sufficiently large $r$, then $(V,\mu)$ is $L^p$-parabolic.
\end{theorem}

To reach the critical logarithmic scale for every $1<p<\infty$, we impose three standard geometric assumptions.

The two-sided Green-function estimate below resolves a periodicity
obstruction in the graph setting. Indeed, $\mathrm{(VD)}$, $\mathrm{(PI)}$,
and $(P_0)$ do not exclude periodic graphs, and the Gaussian lower estimate
in \cite{De} cannot be applied directly when the original graph has no
loops. We therefore introduce the auxiliary transition operator
$\widehat P=(P+P^2)/2$. The associated weighted graph retains the original
edges, adds the two-step connections and a positive loop at every vertex,
preserves the vertex measure and the relevant geometric assumptions, and
satisfies $\frac12 g\leq\widehat g\leq g$. Transferring the Gaussian
estimates from the auxiliary graph back to the original one gives the
following result without imposing a loop condition on the original graph.

\begin{theorem}\label{GF-est}\rm

Assume that $(V,\mu)$ satisfies $\mathrm{(VD)}$, $\mathrm{(PI)}$, and $(P_0)$. Then
\begin{equation}\label{green-est-intro}
g(x,y)\simeq
\sum_{n=d(x,y)}^{\infty}\frac{n}{V(x,n)}.
\end{equation}
\end{theorem}

Combining Theorem \ref{GF-est} with the Green kernel characterization of $L^p$-parabolicity first yields an exact series criterion.

\begin{theorem}\label{tvcon2}\rm
	Let $1<p<\infty$, and assume that $(V,\mu)$ satisfies
	$\mathrm{(VD)}$, $\mathrm{(PI)}$, and $(P_0)$.  Then $(V,\mu)$ is
	$L^p$-parabolic if and only if, for some $o\in V$,
	\begin{equation}\label{bdg}
		\sum_{n=0}^\infty
		\left(\sum_{m=n}^\infty\frac{m}{V(o,m)}\right)^q
		V(S(o,n))=\infty,
	\end{equation}
	where $S(o,n)=\{x\in V:d(x,o)=n\}$ and
	$V(S(o,n))=V(o,n)-V(o,n-1)$ for $n\geq1$.
\end{theorem}

Together, the results above address three specific gaps between the
manifold starting point and the graph setting. First, Theorem \ref{tm1}
shows that, on weighted graphs, the nonlinear system \eqref{equ} is
equivalent to $L^p$-parabolicity and to the $L^q$-Liouville property;
Green operators on finite subsets, the estimate in Lemma \ref{lem:green-domination},
and exhaustion replace smooth differential arguments. Second,
Theorem \ref{vol-lp} works without $\mathrm{(VD)}$, $\mathrm{(PI)}$, or
$(P_0)$ by using even-time return probabilities, so its proof remains
valid for periodic chains. Third, Theorems \ref{GF-est} and \ref{tvcon2}
remove the loop obstruction under the three geometric assumptions: the
auxiliary graph provides Gaussian estimates, while the comparison
$\frac12g\leq\widehat g\leq g$ transfers the resulting Green estimate and
exact series criterion back to the original graph. Corollary \ref{co2}
then gives the critical logarithmic sufficient condition.

The examples in Section \ref{examples} show that these criteria recover the exact threshold for graphs of polynomial volume growth. In particular, if a Cayley graph has polynomial growth of degree $D$, then it is $L^p$-parabolic if and only if $p\geq D/2$.

The paper is organized as follows. Section \ref{sec2} develops the killed-walk and Green-operator preliminaries. Section \ref{sec:capacity} proves the variational characterizations of $L^p$-capacity. Section \ref{sec4} proves Theorem \ref{tm1} and records a monotonicity consequence. Section \ref{volume} establishes the volume-growth and Green-function estimates, and Section \ref{examples} applies them to Cayley graphs of polynomial growth.

\vspace*{4pt}\noindent\textbf{Notation.}
The letters $c,c',C,C',C_1,\ldots$ denote positive constants independent of the principal variables and may change from line to line. We write $f\lesssim g$ if $f\leq Cg$, and $f\gtrsim g$ if $f\geq Cg$, for a constant $C>0$ independent of the relevant parameters. Finally, $f\simeq g$ means that both $f\lesssim g$ and $f\gtrsim g$ hold.

\section{Preliminaries}\label{sec2}

The global transition kernel, heat kernel, Laplacian, Green function, and
Green operator were introduced in the Introduction.  In this section we
collect the corresponding Dirichlet objects on a subset of the graph and
record the elementary identities used in the capacity and Liouville
arguments below.  For background on killed random walks and Green kernels,
we refer to \cite[Sections~1.5--1.6]{MBbook}.

\subsection{Killed random walks and the Dirichlet Laplacian}

Let $U\subset V$, and identify every function on $U$ with its zero
extension to $V\setminus U$.  Denote by
\[
\tau_U:=\inf\{n\geq0:X_n\notin U\}
\]
the first exit time from $U$.  The transition probabilities of the walk
killed upon exiting $U$ are
\[
P_n^U(x,y)
:=
\mathbb P_x[X_n=y,\ n<\tau_U],
\qquad x,y\in V,
\]
where $P_n^U(x,y)=0$ whenever $x\notin U$ or $y\notin U$.  The associated
killed heat kernel is
\[
p_n^U(x,y):=\frac{P_n^U(x,y)}{\mu(y)}.
\]
Reversibility gives
\[
p_n^U(x,y)=p_n^U(y,x),
\qquad x,y\in V,
\]
and, when $U=V$, one has $P_n^V=P_n$ and $p_n^V=p_n$.

Let $I_Uf:=\mathbf 1_Uf$.  For a function $f$ on $U$, define
\[
P_n^Uf(x):=\sum_{y\in U}P_n^U(x,y)f(y),
\qquad
P^U:=P_1^U,
\]
and set
\[
\Delta_U:=P^U-I_U.
\]
Equivalently, $P^U=I_UPI_U$ and $\Delta_U=I_U\Delta I_U$.  In particular,
if $f$ is extended by zero outside $U$, then
\[
\Delta_Uf(x)=\Delta f(x),
\qquad x\in U.
\]
This identity is asserted only on $U$; the full graph Laplacian of the zero
extension need not vanish on $V\setminus U$.

The Markov property yields the following standard semigroup identities.

\begin{proposition}\label{pro2-2}\rm
For all $m,n\geq0$ and $x,y\in V$,
\[
P_{n+m}^U(x,y)
=\sum_{z\in U}P_n^U(x,z)P_m^U(z,y),\]
and
\[P_0^U =I_U,
\quad
P_n^U=(P^U)^n.\]
\end{proposition}

\subsection{Dirichlet Green functions and Green operators}

The Dirichlet Green function on $U$ is
\[
g^U(x,y):=\sum_{n=0}^{\infty}p_n^U(x,y),
\qquad x,y\in V,
\]
where the value $+\infty$ is allowed and $g^U(x,y)=0$ if either variable
lies outside $U$.  In particular, write
\[
g(x,y)=g^V(x,y).
\]
The following finiteness statement is standard; see
\cite[Theorem~1.31]{MBbook}.
\begin{proposition}\label{gu-lemma}\rm
Assume either that $(V,\mu)$ is transient or that $U\neq V$.  Then
\[
g^U(x,y)<\infty,
\qquad x,y\in U.
\]
If $x$ and $y$ lie in the same connected component of $U$, then
$g^U(x,y)>0$.
\end{proposition}

For $f\in\ell^+(U)$, define the Green operator on $U$ by
\[
G^Uf(x)
:=
\sum_{y\in U}g^U(x,y)f(y)\mu(y),
\qquad x\in V,
\]
allowing the value $+\infty$.  By convention, $G^Uf=0$ on $V\setminus U$.
For a signed function $f$, we use the same formula whenever the defining
series is absolutely convergent at the points under consideration; this
includes finitely supported $f$ whenever $g^U$ is finite.  We do not
introduce a separate symbol for this natural domain.
For every $f\in\ell^+(U)$, Tonelli's theorem and
Proposition~\ref{pro2-2} give
\[
G^Uf(x)
=
\sum_{n=0}^{\infty}P_n^Uf(x),
\qquad x\in U,
\]
with both sides taking value in $[0,+\infty]$.
Symmetry of $g^U$ and Tonelli's theorem also yield
\[
\sum_{x\in U}G^Uf(x)h(x)\mu(x)
=
\sum_{x\in U}f(x)G^Uh(x)\mu(x),
\qquad f,h\in\ell^+(U),
\]
 The same identity holds for
signed functions whenever the corresponding double series is absolutely
convergent.

The next two lemmas record the inverse and domination properties of
$-\Delta_U$ and $G^U$ needed below.
\begin{lemma}\label{lem2}\rm
Let $f\in\ell^+(U)$ and assume that $G^Uf(x)<\infty$ for every $x\in U$.
Then
\begin{equation}\label{eq_lem5}
-\Delta_U(G^Uf)=f
\qquad\text{on }U.
\end{equation}
In particular, whenever $g^U$ is finite, for every $y\in U$,
\[
-\Delta_U g^U(\cdot,y)
=
\frac{1}{\mu(y)}\mathbf 1_{\{y\}}
\qquad\text{on }U.
\]
\end{lemma}

\begin{proof}
Fix $x\in U$.  Since the graph is locally finite and $G^Uf$ is finite on
$U$, the Markov property and the nonnegativity of all terms give
\begin{align*}
P^UG^Uf(x)
&=\sum_{y\in U}P_1^U(x,y)\sum_{n=0}^{\infty}P_n^Uf(y)\\
&=\sum_{n=0}^{\infty}P_{n+1}^Uf(x)
 =G^Uf(x)-f(x).
\end{align*}
Because $-\Delta_U=I_U-P^U$, this proves \eqref{eq_lem5}.  The last
assertion follows by taking
$f=\mu(y)^{-1}\mathbf 1_{\{y\}}$.
\end{proof}

The following elementary inequality will be
used in the proof of the main equivalence theorem.  We include its short
semigroup proof for completeness.

\begin{lemma}\label{lem:green-domination}\rm
Let $U\subset V$ be connected and assume that $g^U$ is finite.  Let
$u\in\ell^+(U)$, extended by zero to $V\setminus U$, satisfy
\[
-\Delta u=-\Delta_Uu\geq0
\qquad\text{on }U.
\]
Then
\[
G^U(-\Delta u)(x)\leq u(x),
\qquad x\in U.
\]
\end{lemma}

\begin{proof}
Set $v:=-\Delta_Uu\geq0$.  For every integer $N\geq0$,
\begin{align*}
\sum_{n=0}^{N}P_n^Uv
&=\sum_{n=0}^{N}P_n^U(I_U-P^U)u\\
&=u-P_{N+1}^Uu
\leq u.
\end{align*}
By the nonnegative semigroup representation of $G^U$ and monotone
convergence,
\[
G^Uv(x)
=\lim_{N\to\infty}\sum_{n=0}^{N}P_n^Uv(x)
\leq u(x),
\qquad x\in U.
\]
Since the zero extension satisfies $v=-\Delta u$ on $U$, the conclusion
follows.
\end{proof}

We shall also use the elementary domain monotonicity of Dirichlet Green
functions.  If $U_1\subset U_2$, then
\[
0\leq g^{U_1}(x,y)\leq g^{U_2}(x,y),
\qquad x,y\in V.
\]
Moreover, if $U_k\uparrow U$, then
\[
g^{U_k}(x,y)\uparrow g^U(x,y),
\qquad x,y\in V,
\]
because the events $\{n<\tau_{U_k}\}$ increase to $\{n<\tau_U\}$ for
each fixed $n$.

\subsection{The \texorpdfstring{$L^q$}{Lq}-Green kernel}

Having established the Dirichlet Green identities used in exhaustion, we
return to the global kernel that detects $L^p$-parabolicity. Let
$1<q<\infty$.  The $L^q$-Green kernel is defined by
\begin{equation}\label{gq}
g_q(x,y)
:=
\sum_{z\in V}g(x,z)g(z,y)^{q-1}\mu(z),
\qquad x,y\in V.
\end{equation}
On the diagonal, symmetry of $g$ gives the important identity
\[
g_q(x,x)
=
\sum_{z\in V}g(x,z)^q\mu(z)
=
\|g(x,\cdot)\|_{L^q(V)}^q.
\]
If $(V,\mu)$ is recurrent, then $g\equiv+\infty$ and hence
$g_q\equiv+\infty$.  In the transient case, the following lemma shows that
the finiteness of $g_q$ is independent of the two chosen poles.

\begin{lemma}\label{lem1}\rm
If $g_q(x_0,y_0)<\infty$ for some $(x_0,y_0)\in V\times V$, then
\[
g_q(x,y)<\infty
\qquad\text{for every }(x,y)\in V\times V.
\]
Equivalently, $g_q$ is either finite at every pair of vertices or infinite
at every pair.
\end{lemma}

\begin{proof}
Fix $x,y\in V$.  Since the graph is connected, there exist integers
$i,j\geq0$ such that
\[
a:=P_i(x_0,x)>0,
\qquad
b:=P_j(y_0,y)>0.
\]
The Chapman--Kolmogorov identity implies, for every $z\in V$,
\[
g(x_0,z)\geq a\,g(x,z).
\]
Using symmetry of the Green function in the same way gives
\[
g(z,y_0)=g(y_0,z)
\geq b\,g(y,z)
=b\,g(z,y).
\]
Substitution into \eqref{gq} yields
\[
g_q(x_0,y_0)
\geq
ab^{q-1}g_q(x,y).
\]
Thus $g_q(x,y)<\infty$.  Since $x$ and $y$ were arbitrary, the first
assertion follows; interchanging the two pairs gives the equivalent
dichotomy.
\end{proof}

\section{\texorpdfstring{$L^p$}{Lp}-capacity}\label{sec:capacity}

The local capacity itself is defined for every $U\subset V$.  Whenever
we invoke finiteness of the Dirichlet Green function or its variational
characterization below, we assume that either $U\neq V$ or $(V,\mu)$ is
transient; see Proposition \ref{gu-lemma}.  This hypothesis is automatic
for the finite domains used in exhaustion arguments.

\begin{definition}\label{lpcap}\rm
Let $K\subset U$ be finite and let $1\leq p<\infty$.  The $L^p$-capacity
of $(K,U)$ is defined by
\begin{equation}
C_p(K,U)
=\sup\left\{\nu(K)^p:\ \nu=f\mu,\ f\in\ell^+(K),\
\|G^Uf\|_{L^q(U)}\leq1\right\},
\end{equation}
where $q$ is the H\"older conjugate exponent of $p$.  When $U=V$, we
write $C_p(K):=C_p(K,V)$; see \eqref{CpK}.
\end{definition}

The homogeneity in this definition will be used repeatedly.  More
precisely, if $\nu=f\mu$ and $f\in\ell^+(K)$, then
\begin{equation}\label{rec-eq}
\inf_{\substack{f\in\ell^+(K)\\ \nu(K)>0}}
\frac{\|G^Uf\|_{L^q(U)}}{\nu(K)}
=\frac{1}{C_p(K,U)^{1/p}},
\end{equation}
with the convention that both sides are $+\infty$ when $C_p(K,U)=0$.
Indeed, rescaling any nonzero $f$ by $\|G^Uf\|_{L^q(U)}^{-1}$ shows that
\[
C_p(K,U)^{1/p}
=\sup_{\substack{f\in\ell^+(K)\\ \nu(K)>0}}
\frac{\nu(K)}{\|G^Uf\|_{L^q(U)}},
\]
and taking reciprocals gives \eqref{rec-eq}.

\begin{remark}\rm
There are several other notions of capacity in the literature.  For
example, after extending $\varphi\in\ell_0(U)$ by zero to $V\setminus U$,
the harmonic capacity is
\[
\operatorname{Cap}(K,U)
=\inf\left\{
\frac12\sum_{x,y\in V}(\varphi(x)-\varphi(y))^2\mu_{xy}:
\varphi\in\ell_0(U),\ \varphi\geq1\ \text{on }K
\right\}.
\]
When $U=V$, we write $\operatorname{Cap}(K):=\operatorname{Cap}(K,V)$.
For every finite $K\subset U$, with $U$ being finite or $U=V$, one has
\[
\operatorname{Cap}(K,U)=C_1(K,U);
\]
see \cite[Theorem~7.8 and Proposition~7.9]{MBbook}.
\end{remark}

To connect this dual definition to the Laplacian estimates used later,
we introduce two variational characterizations of $C_p(K,U)$ for
$1<p<\infty$.  For a
finite set $K\subset U$, define
\begin{equation*}
\widehat C_p(K,U)
=\inf\left\{\|f\|_{L^p(U)}^p:
 f\in L_+^p(U),\ G^Uf\geq1\ \text{on }K\right\},
\end{equation*}
and
\begin{equation*}
\overline C_p(K,U)
=\inf\left\{\|\Delta\varphi\|_{L^p(U)}^p:
 \varphi\in\ell_0(U),\ \varphi\geq1\ \text{on }K\right\}.
\end{equation*}
\begin{theorem}\label{cape}\rm
Let $1<p<\infty$ and let $K$ be a finite subset of $U$.  Assume that
either $U\neq V$ or $(V,\mu)$ is transient.  Then
\begin{equation}\label{eq-cap}
C_p(K,U)=\widehat C_p(K,U)=\overline C_p(K,U).
\end{equation}
\end{theorem}

\begin{proof}
We first prove $C_p(K,U)=\widehat C_p(K,U)$.  The only point requiring
some care is that the Green operator may take the value $+\infty$.  We
therefore apply the minimax argument first to truncated Green operators.
For $N\geq0$, set
\[
G_N^U=\sum_{n=0}^N(P^U)^n.
\]
Let
\begin{align*}
X&=\{\nu\in\ell^+(U):\ \nu(K)=1,\ \nu=0\ \text{on }U\setminus K\},\\
Y&=\{u\in L_+^p(U):\ \|u\|_{L^p(U)}\leq1\},
\end{align*}
and, for $(\nu,u)\in X\times Y$, define
\[
\mathcal E_N(\nu,u)=\sum_{x\in U}G_N^Uu(x)\nu(x).
\]
If $f_\nu(x)=\nu(x)/\mu(x)$, the symmetry of the truncated Green kernel
gives
\[
\mathcal E_N(\nu,u)
=\sum_{y\in U}G_N^Uf_\nu(y)u(y)\mu(y).
\]
Since $K$ is finite and the graph is locally finite,
$G_N^Uf_\nu$ has finite support.  Hence H\"older duality yields
\[
\sup_{u\in Y}\mathcal E_N(\nu,u)
=\|G_N^Uf_\nu\|_{L^q(U)}.
\]
For fixed $u\in Y$, the minimum over all probability measures supported
on $K$ is attained at a point of $K$, and therefore
\[
\inf_{\nu\in X}\mathcal E_N(\nu,u)
=\min_{x\in K}G_N^Uu(x).
\]
The set $X$ is a compact convex subset of $\mathbb R^{|K|}$, the set
$Y$ is convex, and $\mathcal E_N$ is bilinear and continuous.  The
minimax theorem \cite[Theorem~2.4.1]{AHbook} gives
\begin{equation}\label{eq:minimax-truncated}
A_N:=\min_{\nu\in X}\|G_N^Uf_\nu\|_{L^q(U)}
=\sup_{u\in Y}\min_{x\in K}G_N^Uu(x)=:B_N.
\end{equation}

We next pass to the full Green operator.  Both $A_N$ and $B_N$ are
nondecreasing.  Put
\[
A=\inf_{\nu\in X}\|G^Uf_\nu\|_{L^q(U)}.
\]
Clearly $A_N\leq A$.  Choose a minimizer $\nu_N\in X$ for $A_N$.
By compactness, a subsequence $\nu_{N_j}$ converges to some
$\nu_\ast\in X$.  For every fixed $M$ and all sufficiently large $j$,
\[
\|G_M^Uf_{\nu_{N_j}}\|_{L^q(U)}
\leq A_{N_j}.
\]
Letting $j\to\infty$ and using the continuity of the finite-range
operator $G_M^U$, we obtain
\[
\|G_M^Uf_{\nu_\ast}\|_{L^q(U)}
\leq\lim_{N\to\infty}A_N.
\]
Now let $M\to\infty$.  Monotone convergence gives
\[
A\leq\|G^Uf_{\nu_\ast}\|_{L^q(U)}
\leq\lim_{N\to\infty}A_N.
\]
Thus $A_N\uparrow A$.  Similarly, because $K$ is finite,
$\min_{x\in K}G_N^Uu(x)\uparrow\min_{x\in K}G^Uu(x)$ for every
$u\in Y$.  Hence
\begin{align*}
\lim_{N\to\infty}B_N
&=\sup_{N\geq0}\sup_{u\in Y}\min_{x\in K}G_N^Uu(x)\\
&=\sup_{u\in Y}\min_{x\in K}G^Uu(x).
\end{align*}
Passing to the limit in \eqref{eq:minimax-truncated}, and then using
\eqref{rec-eq} and homogeneity, we obtain
\[
\frac{1}{C_p(K,U)^{1/p}}
=\frac{1}{\widehat C_p(K,U)^{1/p}}.
\]
This proves the first equality in \eqref{eq-cap}, including the case in
which the common reciprocal is $+\infty$.

We now prove $\widehat C_p(K,U)=\overline C_p(K,U)$.  Let
$f\in L_+^p(U)$ satisfy $G^Uf\geq1$ on $K$, and fix
$\varepsilon>0$.  Choose an increasing sequence of finite sets
$F_j\uparrow U$ and put $f_j=f\mathbf1_{F_j}$.  Since
$G^Uf_j\uparrow G^Uf$ and $K$ is finite, for all sufficiently large
$j$,
\begin{equation}\label{eq:finite-source-admissible}
(1+\varepsilon)G^Uf_j\geq1\qquad\text{on }K.
\end{equation}
Fix such a $j$ and write $T=P^U$.  For $N\geq0$, define the Ces\`aro
truncation
\[
v_N=(1+\varepsilon)\sum_{n=0}^N
\left(1-\frac{n}{N+1}\right)T^nf_j.
\]
Local finiteness implies $v_N\in\ell_0(U)$.  Moreover,
$v_N\uparrow(1+\varepsilon)G^Uf_j$ pointwise, so
\eqref{eq:finite-source-admissible} implies that $v_N\geq1$ on $K$ for
all sufficiently large $N$.  A direct telescoping computation gives
\begin{equation}\label{eq:cesaro-laplacian}
-\Delta_Uv_N=(I_U-T)v_N
=(1+\varepsilon)\bigl(f_j-A_{N+1}f_j\bigr),
\end{equation}
where
\[
A_{N+1}f_j=\frac{1}{N+1}\sum_{n=1}^{N+1}T^nf_j.
\]

The operator $T$ is a contraction on $L^p(U)$, and its fixed-point space
in $L^p(U)$ is trivial.  Indeed, suppose that $Th=h$.  Extend $h$ by
zero to $V\setminus U$.  Jensen's inequality and reversibility give
\[
|h|^p=|Th|^p\leq T(|h|^p),
\qquad
\sum_{x\in U}T(|h|^p)(x)\mu(x)
\leq\sum_{x\in U}|h(x)|^p\mu(x).
\]
Thus equality holds throughout.  The equality case in the strict
convexity inequality forces $h$ to have the same value at adjacent
vertices, while an edge from $U$ to $V\setminus U$ forces that value to
be zero.  If $U\neq V$, every connected component of $U$ has such a
boundary edge, and hence $h=0$.  If $U=V$, connectivity makes $h$
constant; transience implies $\mu(V)=\infty$, so an $L^p$ constant is
again zero.  By the mean ergodic theorem on the reflexive space
$L^p(U)$, it follows that
\[
A_{N+1}f_j\longrightarrow0
\qquad\text{in }L^p(U).
\]
Using \eqref{eq:cesaro-laplacian}, the identity
$\Delta_Uv_N=\Delta v_N$ on $U$, and the admissibility of $v_N$, we
therefore obtain
\[
\overline C_p(K,U)
\leq(1+\varepsilon)^p\|f_j\|_{L^p(U)}^p
\leq(1+\varepsilon)^p\|f\|_{L^p(U)}^p.
\]
Taking the infimum over all admissible $f$ and then letting
$\varepsilon\to0$ proves
\begin{equation}\label{dir-1}
\overline C_p(K,U)\leq\widehat C_p(K,U).
\end{equation}

For the reverse inequality, let $\varphi\in\ell_0(U)$ satisfy
$\varphi\geq1$ on $K$.  Since $\varphi$ has finite support, so does
$(I_U-P^U)\varphi$.  For every $N$,
\[
\sum_{n=0}^N(P^U)^n(I_U-P^U)\varphi
=\varphi-(P^U)^{N+1}\varphi.
\]
For each fixed $x\in U$, the last term tends to zero: if
$F=\operatorname{supp}\varphi$, then
\[
|(P^U)^{N+1}\varphi(x)|
\leq\sum_{y\in F}P_{N+1}^U(x,y)|\varphi(y)|,
\]
and every sequence $P_n^U(x,y)$ tends to zero because its sum is finite.
Consequently,
\[
G^U(-\Delta_U\varphi)=\varphi.
\]
Set $u=|\Delta_U\varphi|$.  Positivity of the Green kernel gives
\[
G^Uu\geq G^U(-\Delta_U\varphi)=\varphi\geq1
\qquad\text{on }K.
\]
Since $\Delta_U\varphi=\Delta\varphi$ on $U$, we conclude that
\[
\widehat C_p(K,U)
\leq\|u\|_{L^p(U)}^p
=\|\Delta\varphi\|_{L^p(U)}^p.
\]
Taking the infimum over $\varphi$ yields
\begin{equation}\label{dir-2}
\widehat C_p(K,U)\leq\overline C_p(K,U).
\end{equation}
Combining \eqref{dir-1} and \eqref{dir-2} completes the proof.
\end{proof}

\begin{remark}\label{rem-p par}\rm
	For $p=1$, the first equality in \eqref{eq-cap} does not hold in
	general, whereas
	the second variational description remains meaningful.  For
	$1<p<\infty$, on a transient graph the capacity
	$C_p(K,V)=\overline C_p(K,V)$ is bounded above by the
	$p$-capacity $\operatorname{Cap}_p(K)$ defined in \eqref{pcap}.  Indeed,
	Jensen's inequality gives
	\begin{align*}
		\|\Delta\varphi\|_{L^p(V)}^p
		&=\sum_{x\in V}\left|\sum_{y\sim x}
		\frac{\mu_{xy}}{\mu(x)}(\varphi(y)-\varphi(x))\right|^p\mu(x)\\
		&\leq\sum_{x,y\in V}\mu_{xy}|\varphi(y)-\varphi(x)|^p.
	\end{align*}
	Consequently, $p$-parabolicity implies $L^p$-parabolicity in the
	transient case.  The same conclusion is automatic in the recurrent
	case, because then $Gf\equiv+\infty$ for every nonzero, nonnegative,
	finitely supported $f$.
\end{remark}

We next record the monotonicity, subadditivity, and exhaustion properties
that will be used in Section \ref{sec4}.

\begin{proposition}\label{prop-cap-1}\rm
Let $1<p<\infty$, and let finite sets satisfy
$K_1\subset K_2\subset U_1\subset U_2$.  Assume that either $U_2\neq V$
or $(V,\mu)$ is transient.  Then
\begin{equation}\label{capm}
C_p(K_1,U_1)\leq C_p(K_2,U_1),
\qquad
C_p(K_1,U_2)\leq C_p(K_1,U_1),
\end{equation}
and, for finite $K_1,K_2\subset U$, where either $U\neq V$ or
$(V,\mu)$ is transient,
\begin{equation}\label{capu}
C_p(K_1\cup K_2,U)
\leq C_p(K_1,U)+C_p(K_2,U).
\end{equation}
\end{proposition}

\begin{proof}
The first inequality in \eqref{capm} follows because every function
admissible for $\widehat C_p(K_2,U_1)$ is also admissible for
$\widehat C_p(K_1,U_1)$.  For the second inequality, domain monotonicity
gives $g^{U_1}\leq g^{U_2}$.  Thus, if a measure supported on $K_1$ is
admissible in the definition of $C_p(K_1,U_2)$, then it is also
admissible for $C_p(K_1,U_1)$.  This proves both monotonicity assertions.

To prove \eqref{capu}, fix $\varepsilon>0$ and choose
$f_i\in L_+^p(U)$ such that $G^Uf_i\geq1$ on $K_i$ and
\[
\|f_i\|_{L^p(U)}^p
\leq C_p(K_i,U)+\frac{\varepsilon}{2},
\qquad i=1,2.
\]
For $f=\max\{f_1,f_2\}$, one has $G^Uf\geq1$ on $K_1\cup K_2$ and,
pointwise, $f^p\leq f_1^p+f_2^p$.  Therefore
\[
\widehat C_p(K_1\cup K_2,U)
\leq\|f\|_{L^p(U)}^p
\leq\|f_1\|_{L^p(U)}^p+\|f_2\|_{L^p(U)}^p.
\]
Letting $\varepsilon\to0$ and applying Theorem \ref{cape} proves the
claim.
\end{proof}

\begin{proposition}\label{prop-cap-2}\rm
Let $\{U_n\}$ be an increasing exhaustion of $U$, let $K\subset U_1$
be finite, and assume that either $U\neq V$ or $(V,\mu)$ is transient.
Then, for every $1<p<\infty$,
\[
\lim_{n\to\infty}C_p(K,U_n)=C_p(K,U).
\]
\end{proposition}

\begin{proof}
By \eqref{capm}, the sequence $C_p(K,U_n)$ is decreasing and bounded
below by $C_p(K,U)$.  Let its limit be $L$.  It remains to prove
$L\leq C_p(K,U)$.

Take $f\in L_+^p(U)$ with $G^Uf\geq1$ on $K$, and fix
$\varepsilon>0$.  Put
\[
f_{n,\varepsilon}=(1+\varepsilon)f\mathbf1_{U_n}.
\]
The domain monotonicity of the killed kernels and monotone convergence
give, for every $x\in K$,
\[
G^{U_n}f_{n,\varepsilon}(x)
\longrightarrow(1+\varepsilon)G^Uf(x)
\geq1+\varepsilon.
\]
Since $K$ is finite, $G^{U_n}f_{n,\varepsilon}\geq1$ on all of $K$ for
one and the same sufficiently large $n$.  Hence Theorem \ref{cape}
gives
\[
C_p(K,U_n)=\widehat C_p(K,U_n)
\leq(1+\varepsilon)^p\|f\mathbf1_{U_n}\|_{L^p(U)}^p
\leq(1+\varepsilon)^p\|f\|_{L^p(U)}^p.
\]
Letting $n\to\infty$, taking the infimum over admissible $f$, and then
letting $\varepsilon\to0$, we obtain $L\leq C_p(K,U)$.  This completes
the proof.
\end{proof}

\section{Proof of Theorem \ref{tm1}}\label{sec4}

Although Theorem \ref{tm1} concerns $1<p<\infty$, the case $p=1$ for this
capacity is governed by classical recurrence.  The following standard
equivalences are collected, for example, in
\cite[Theorems~1.16 and 2.12]{WWbook}.

\begin{theorem}\label{tm4-1}\rm
Let $(V,\mu)$ be an infinite, connected, locally finite graph.  The
following statements are equivalent:
\begin{enumerate}
\item[(1)] $(V,\mu)$ is parabolic;
\item[(2)] every nonnegative superharmonic function is constant;
\item[(3)] $g(x,y)=\infty$ for some, equivalently all, $x,y\in V$;
\item[(4)] $\operatorname{Cap}(\{x\})=0$ for some, equivalently every,
$x\in V$.
\end{enumerate}
\end{theorem}

If $(V,\mu)$ is parabolic, then $g\equiv+\infty$, so (d) in Theorem
\ref{tm1} holds.  Moreover, both (b) and (c) follow from Theorem
\ref{tm4-1}, and (a) follows directly from the definition because
$Gf\equiv+\infty$ for every nonzero $f\in\ell^+(K)$.  We may therefore
assume throughout the proof of Theorem \ref{tm1} that $(V,\mu)$ is
transient.

\begin{proof}[Proof of Theorem \ref{tm1}]
By Lemma \ref{lem1}, condition (d) is independent of the chosen pair of
vertices.  We prove that each of (a), (b), and (c) is equivalent to (d).

\medskip\noindent
\textbf{ (a)$\Longleftrightarrow$(d).}
Assume first that (d) fails.  Then, for some $x_0\in V$,
$g_q(x_0,x_0)<\infty$.  Set $h(x)=g(x,x_0)$.  Since
\[
-\Delta h=\frac{1}{\mu(x_0)}\mathbf1_{\{x_0\}},
\]
Green's identity and H\"older's inequality give, for every
$v\in\ell_0(V)$ with $v(x_0)\geq1$,
\begin{align*}
1
&\leq\sum_{x\in V}(-\Delta h(x))v(x)\mu(x)\\
&=\sum_{x\in V}h(x)(-\Delta v(x))\mu(x)\\
&\leq\left(\sum_{x\in V}h(x)^q\mu(x)\right)^{1/q}
       \|\Delta v\|_{L^p(V)}\\
&=g_q(x_0,x_0)^{1/q}\|\Delta v\|_{L^p(V)}.
\end{align*}
Taking the infimum over such $v$ yields
\[
\overline C_p(\{x_0\})
\geq g_q(x_0,x_0)^{-p/q}>0.
\]
By Theorem \ref{cape}, $C_p(\{x_0\})>0$, so the graph is not
$L^p$-parabolic.  This proves (a)$\Rightarrow$(d).

Conversely, assume that the graph is not $L^p$-parabolic.  Then
$C_p(A)>0$ for some finite set $A\subset V$.  If every singleton in
$A$ had zero capacity, the subadditivity in Proposition
\ref{prop-cap-1} would give $C_p(A)=0$.  Hence there exists
$x_0\in A$ such that
\[
C_p(\{x_0\})>0.
\]
Let $U$ be a finite connected set containing $x_0$, and define
\begin{align*}
A_U
&:=G\bigl([g^U(\cdot,x_0)]^{q-1}\bigr)(x_0),\\
f_U(x)
&:=\frac{[g^U(x,x_0)]^{q-1}}{A_U}.
\end{align*}
The function $f_U$ has finite support and $Gf_U(x_0)=1$.  Therefore,
using the variational characterization in Theorem \ref{cape},
\begin{align*}
0<C_p(\{x_0\})
&\leq\|f_U\|_{L^p(V)}^p\\
&=\frac{\sum_{x\in U}[g^U(x,x_0)]^q\mu(x)}{A_U^p}.
\end{align*}
Since $g(x_0,x)\geq g^U(x_0,x)$ and the Green functions are symmetric,
\[
A_U
=\sum_{x\in U}g(x_0,x)[g^U(x,x_0)]^{q-1}\mu(x)
\geq\sum_{x\in U}[g^U(x,x_0)]^q\mu(x).
\]
Consequently,
\[
C_p(\{x_0\})\leq A_U^{1-p},
\qquad
A_U\leq C_p(\{x_0\})^{-1/(p-1)}.
\]
Choose an increasing exhaustion $U\uparrow V$.  By domain monotonicity,
$g^U(\cdot,x_0)\uparrow g(\cdot,x_0)$, and hence monotone convergence
gives
\[
A_U\uparrow
\sum_{x\in V}g(x_0,x)^q\mu(x)=g_q(x_0,x_0).
\]
The preceding uniform bound shows that $g_q(x_0,x_0)<\infty$, which is
the negation of (d).  Thus (d)$\Rightarrow$(a).

\medskip\noindent
\textbf{
(b)$\Longleftrightarrow$(d).}
If (d) fails, then $g(\cdot,x_0)\in L^q(V)$ for some $x_0$, namely
\[
\|g(\cdot,x_0)\|_{L^q(V)}^q=g_q(x_0,x_0)<\infty.
\]
The function $g(\cdot,x_0)$ is nonnegative and superharmonic but not
harmonic, so it is nonconstant. Hence the $L^q$-Liouville property fails.
This proves
(b)$\Rightarrow$(d).

Conversely, suppose that (b) fails, and let $f\in L_+^q(V)$ be a
nonconstant nonnegative superharmonic function.  We first note that $f$
cannot be
harmonic.  Indeed, if $f=Pf$, then strict convexity and Jensen's
inequality give
\[
f(x)^q\leq P(f^q)(x).
\]
After summation against $\mu$, reversibility gives
\[
\sum_{x\in V}P(f^q)(x)\mu(x)
=\sum_{x\in V}f(x)^q\mu(x).
\]
Thus equality holds pointwise in Jensen's inequality.  It follows that
$f$ is constant on adjacent vertices and hence, by connectivity,
constant on $V$.  Since a transient graph has $\mu(V)=\infty$, an
$L^q$ constant must be zero, a contradiction.  Therefore
$-\Delta f(x_0)>0$ for at least one vertex $x_0$; in particular,
$f(x_0)>0$.

Choose $\lambda>0$ so that
$\lambda f(x_0)\geq g(x_0,x_0)$.  Let $U_n\uparrow V$ be a connected
exhaustion by finite sets with $x_0\in U_n$.  On
$U_n\setminus\{x_0\}$, the function
\[
\lambda f-g^{U_n}(\cdot,x_0)
\]
is superharmonic.  It is nonnegative at $x_0$, and it is nonnegative on
the exterior boundary because $g^{U_n}(\cdot,x_0)$ vanishes there.
The minimum principle therefore gives
\[
g^{U_n}(x,x_0)\leq\lambda f(x),
\qquad x\in U_n.
\]
Letting $n\to\infty$ yields $g(x,x_0)\leq\lambda f(x)$ for every
$x\in V$.  Hence
\[
g_q(x_0,x_0)
=\sum_{x\in V}g(x,x_0)^q\mu(x)
\leq\lambda^q\|f\|_{L^q(V)}^q<\infty.
\]
Thus (d) fails, proving (d)$\Rightarrow$(b).

\medskip\noindent
\textbf{
(c)$\Longleftrightarrow$(d).}
Assume first that (d) fails and fix $x_0\in V$.  Lemma \ref{lem1}
shows that $g_q(x,x_0)<\infty$ for every $x\in V$.  Define
\[
U(x)=g_q(x,x_0)
=G\bigl(g(\cdot,x_0)^{q-1}\bigr)(x).
\]
Lemma \ref{lem2} gives
\[
-\Delta U=g(\cdot,x_0)^{q-1}>0.
\]
Because $(q-1)(p-1)=1$ and $\Delta U\leq0$,
\[
|\Delta U|^{p-2}\Delta U
=-\bigl(g(\cdot,x_0)^{q-1}\bigr)^{p-1}
=-g(\cdot,x_0).
\]
Applying $\Delta$ once more, we obtain
\[
\Delta\bigl(|\Delta U|^{p-2}\Delta U\bigr)
=-\Delta g(\cdot,x_0)
=\frac{1}{\mu(x_0)}\mathbf1_{\{x_0\}}\geq0.
\]
Thus $U$ is a positive nonharmonic solution of \eqref{equ}, so (c)
fails.  This proves (c)$\Rightarrow$(d).

Finally, suppose that (c) fails.  Let $h\geq0$ be a nonharmonic solution
of \eqref{equ}, and set
\[
h_1=-\Delta h\geq0,
\qquad
w=h_1^{p-1}.
\]
The second inequality in \eqref{equ} says that $w$ is superharmonic.
Since $h$ is not harmonic, $w$ is not identically zero.  The minimum
principle then implies $w>0$ on $V$.

Fix $x_0\in V$ and put
\[
c_0=\frac{w(x_0)}{g(x_0,x_0)}>0.
\]
Let $U_n\uparrow V$ be a connected exhaustion by finite sets containing
$x_0$.  On
$U_n\setminus\{x_0\}$, the function
$w-c_0g^{U_n}(\cdot,x_0)$ is superharmonic.  At $x_0$ it is
nonnegative because $g^{U_n}(x_0,x_0)\leq g(x_0,x_0)$, and it is
nonnegative on the exterior boundary.  The minimum principle gives
\[
w(x)\geq c_0g^{U_n}(x,x_0),
\qquad x\in U_n.
\]
Letting $n\to\infty$, we obtain
\begin{equation}\label{eq:w-green-lower}
h_1(x)^{p-1}=w(x)\geq c_0g(x,x_0),
\qquad x\in V.
\end{equation}

By Lemma \ref{lem:green-domination},
\[
Gh_1(x)=G(-\Delta h)(x)\leq h(x),
\qquad x\in V.
\]
Evaluating at $x_0$ and using \eqref{eq:w-green-lower}, together with
symmetry of $g$, gives
\begin{align*}
h(x_0)
&\geq\sum_{y\in V}g(x_0,y)h_1(y)\mu(y)\\
&\geq c_0^{1/(p-1)}
\sum_{y\in V}g(x_0,y)g(y,x_0)^{1/(p-1)}\mu(y)\\
&=c_0^{1/(p-1)}
\sum_{y\in V}g(y,x_0)^q\mu(y)\\
&=c_0^{1/(p-1)}g_q(x_0,x_0).
\end{align*}
Since $h(x_0)<\infty$, this proves $g_q(x_0,x_0)<\infty$.  Hence (d)
fails, proving (d)$\Rightarrow$(c) and completing the proof.
\end{proof}

\begin{remark}\label{rem-sobolev-liouville}\rm
The Liouville formulation also has a direct Sobolev-space interpretation
under $(P_0)$. Badr and Russ \cite{BR} developed Sobolev interpolation
and related analytic estimates on graphs, while Shao, Yang, and Zhao
\cite{SYZ2} treat Sobolev spaces on general locally finite graphs.
For $1<q<\infty$, define
\[
W^{1,q}(V):=\{u\in L^q(V):|\nabla u|\in L^q(V)\}.
\]
Here
\[
|\nabla u|(x)
:=\left(\frac{1}{2\mu(x)}
\sum_{y\sim x}\mu_{xy}|u(y)-u(x)|^2\right)^{1/2}.
\]
	If $(P_0)$ holds, then $W^{1,q}(V)=L^q(V)$. To see this, fix $x\in V$ and set
$a_y:=|u(y)-u(x)|$ for $y\sim x$.  By $(P_0)$,
\[
\left(\sum_{y\sim x}P(x,y)a_y^2\right)^{q/2}
\leq \max_{y\sim x}a_y^q
\leq \frac1\alpha
\sum_{y\sim x}P(x,y)a_y^q.
\]
Consequently, using reversibility and the inequality
$|a-b|^q\leq2^{q-1}(|a|^q+|b|^q)$, we obtain
\begin{align*}
	\|\nabla u\|_{L^q(V)}^q
	&\leq \frac{1}{2^{q/2}\alpha}
	\sum_{x,y\in V}\mu(x)P(x,y)|u(y)-u(x)|^q\\
	&\leq \frac{2^{q/2}}{\alpha}\|u\|_{L^q(V)}^q.
\end{align*}
Thus
\[
\|\nabla u\|_{L^q(V)}
\leq \sqrt{2}\,\alpha^{-1/q}\|u\|_{L^q(V)}.
\]
The reverse inclusion is immediate from the definition of
$W^{1,q}(V)$.  Consequently, under $(P_0)$, the $L^q$-Liouville
	property in Theorem \ref{tm1} is equivalent to the assertion that
every nonnegative superharmonic function
in $W^{1,q}(V)$ is constant.
\end{remark}
The following monotonicity with respect to the exponent is an immediate
consequence of the Green kernel characterization.

\begin{corollary}\label{st-parabolic}\rm
Let $1\leq s<t<\infty$.  If $(V,\mu)$ is $L^s$-parabolic, then it is
$L^t$-parabolic.
\end{corollary}

\begin{proof}
Let $s'$ and $t'$ be the H\"older conjugates of $s$ and $t$,
respectively.  If $s=1$, the identity $C_1=\operatorname{Cap}$ recalled
after Definition \ref{lpcap} and Theorem \ref{tm4-1} show that the graph
is recurrent; the conclusion then follows from Theorem \ref{tm1}.
Assume now that $1<s<t<\infty$.  If the
graph is recurrent, there is nothing to prove, so assume it is transient.
For a fixed $x_0\in V$, the hitting-time identity gives
\[
g(x,x_0)\leq g(x_0,x_0),
\qquad x\in V;
\]
see \cite[Theorem~1.31]{MBbook}.  Since $s'>t'$, pointwise we have
\[
g(x,x_0)^{s'}
\leq g(x_0,x_0)^{s'-t'}g(x,x_0)^{t'}.
\]
After summation against $\mu$, this becomes
\[
g_{s'}(x_0,x_0)
\leq g(x_0,x_0)^{s'-t'}g_{t'}(x_0,x_0).
\]
If the graph is $L^s$-parabolic, Theorem \ref{tm1} implies that the
left-hand side is infinite.  Therefore $g_{t'}(x_0,x_0)=\infty$, and a
second application of Theorem \ref{tm1} proves
$L^t$-parabolicity.
\end{proof}

\section{Volume-growth criteria}\label{volume}

The classical geometric background begins with surface and manifold
parabolicity \cite{Ahlfors,Tro}; for networks, Yamasaki \cite{Yam} and
Soardi \cite{Sor} developed corresponding criteria, including links with
polynomial growth. In the weighted-graph setting, we use the following
sufficient condition from Grigor'yan \cite[Theorem~6.13]{AGbook}, which is
a consequence of the Nash--Williams test: if,
for some $o\in V$,
\begin{equation}\label{votest}
\sum_{n=1}^\infty\frac{n}{V(o,n)}=\infty,
\end{equation}
then $(V,\mu)$ is parabolic. If
$(V,\mu)$ is $L^1$-parabolic, Corollary \ref{st-parabolic} shows that it
is also $L^p$-parabolic for every $p\geq1$.  Thus, whenever a finite
Green function is required below, we may assume without loss of
generality that $(V,\mu)$ is nonparabolic.

We now pass from this classical endpoint to $p>1$, beginning with the
weak-assumption range $1<p\leq2$.

\begin{proof}[Proof of Theorem \ref{vol-lp}]
If the graph is recurrent, the conclusion follows from Theorem
\ref{tm4-1} and Corollary \ref{st-parabolic}. We may therefore assume
that the graph is transient.
Let $q=p/(p-1)$.  Since $1<p\leq2$, one has $q-1\geq1$.  Using the
symmetry of $g$, Tonelli's theorem, and the identity
$g(o,x)=\sum_{n\geq0}p_n(o,x)$, we obtain
\begin{align}
 g_q(o,o)
&=\sum_{x\in V}g(o,x)g(x,o)^{q-1}\mu(x)\nonumber\\
&=\sum_{n=0}^\infty\sum_{x\in V}
 p_n(o,x)g(x,o)^{q-1}\mu(x).                       \label{est-1a}
\end{align}
For each $n$, the measure $p_n(o,x)\mu(x)$ is a probability measure.
Jensen's inequality therefore gives
\begin{align}
\sum_{x\in V}p_n(o,x)g(x,o)^{q-1}\mu(x)
&\geq\left(\sum_{x\in V}p_n(o,x)g(x,o)\mu(x)\right)^{q-1}\nonumber\\
&=\left(\sum_{m=0}^\infty p_{n+m}(o,o)\right)^{q-1}\nonumber\\
&=\left(\sum_{m=n}^\infty p_m(o,o)\right)^{q-1}.   \label{est-1b}
\end{align}
Here the semigroup identity and reversibility were used in the second
line.  Combining \eqref{est-1a} and \eqref{est-1b}, we find
\begin{equation}\label{est-1}
g_q(o,o)
\geq\sum_{n=0}^\infty
\left(\sum_{m=n}^\infty p_m(o,o)\right)^{q-1}.
\end{equation}

The fixed-point estimate proved in the proof of
\cite[Theorem~3]{L} applies to the self-adjoint Markov contraction $P$,
whose kernel has propagation one.  In the present notation, that estimate
states that
\[
p_{2m}(o,o)\geq\frac{1}{4V(o,r_m)}
\]
whenever $r_m$ is chosen so that
\[
V(o,1)^{-1/2}\exp\left(-\frac{r_m^2}{2m}\right)
V(o,m)^{1/2}\leq\frac12.
\]
The measure lower bound and the infinitude of $V$ ensure that the ambient
measure is infinite, as required there.  Moreover,
\eqref{vol-growth-intro} makes $V(o,m)$ at most polynomial in $m$.
Consequently $r_m=\sqrt{C m\log m}$ is admissible for a sufficiently
large constant $C$ and all $m\geq n_0$.  Thus, with $c=1/4$,
\begin{equation}\label{eq:even-diagonal-lower}
p_{2m}(o,o)
\geq\frac{c}{V(o,\sqrt{C m\log m})},
\qquad m\geq n_0.
\end{equation}
The use of even times is essential here: on a bipartite graph the odd
return probabilities may vanish.  By \eqref{vol-growth-intro}, for all sufficiently
large $m$,
\[
V(o,\sqrt{C m\log m})
\lesssim m^p(\log m)^{p-1},
\]
and hence
\begin{equation}\label{eq:even-return-power}
p_{2m}(o,o)\gtrsim m^{-p}(\log m)^{1-p}.
\end{equation}
For large $n$, restricting the inner sum in \eqref{est-1} to even times
gives
\begin{align*}
\sum_{k=n}^\infty p_k(o,o)
&\geq\sum_{m\geq\lceil n/2\rceil}p_{2m}(o,o)\\
&\gtrsim\sum_{m\geq\lceil n/2\rceil}
 m^{-p}(\log m)^{1-p}\\
&\simeq n^{1-p}(\log n)^{1-p}.
\end{align*}
Since $q-1=1/(p-1)$, substitution into \eqref{est-1} yields
\[
g_q(o,o)
\gtrsim\sum_{n\geq n_0}
\left(n^{1-p}(\log n)^{1-p}\right)^{1/(p-1)}
=\sum_{n\geq n_0}\frac{1}{n\log n}=\infty.
\]
Theorem \ref{tm1} now implies that $(V,\mu)$ is
$L^p$-parabolic.
\end{proof}

\begin{remark}\rm
The loss in the volume condition \eqref{vol-growth-intro} comes from the
logarithmic enlargement of the ball in
\eqref{eq:even-diagonal-lower}. If one instead has the stronger diagonal
lower bound
\[
p_{2n}(o,o)\gtrsim\frac{1}{V(o,\sqrt n)},
\]
then the same argument improves \eqref{vol-growth-intro} to
\[
V(o,r)\lesssim r^{2p}(\log r)^{p-1}.
\]
\end{remark}

In the remainder of this section, we establish the two-sided Green
function estimate in Theorem \ref{GF-est}.  It yields the exact series
criterion in Theorem \ref{tvcon2}, the sufficient condition in Theorem
\ref{vcond}, and the critical logarithmic volume bound in Corollary
\ref{co2}.

\begin{definition}\rm
\phantomsection\label{D}
We say that $(V,\mu)$ satisfies condition $(\Delta)$ if it satisfies
$(P_0)$ and every vertex has a loop; equivalently,
\[
\left\{
\begin{array}{ll}
y\sim x\Longrightarrow \mu_{xy}/\mu(x)\geq\alpha,\\
x\sim x\quad\text{for every }x\in V.
\end{array}
\right.
\tag{$\Delta$}
\]
\end{definition}

The following equivalence, including the two-sided Gaussian heat-kernel
estimate, is due to Delmotte \cite[Theorem~1.7]{De}.  Delmotte states the
estimate for the transition probabilities $P_n(x,y)$; division by
$\mu(y)$ gives the density normalization used in \eqref{GSe}.

\begin{theorem}\label{tmp}\rm
Assume that $(V,\mu)$ satisfies condition $(\Delta)$.  Then the following
conditions are equivalent:
\begin{enumerate}
\item[(A)] $(V,\mu)$ satisfies $\mathrm{(VD)}$ and $\mathrm{(PI)}$;
\item[(B)] there exist positive constants $c_l,C_l,c_r,C_r$ such that
\begin{equation}\label{GSe}
\frac{c_l}{V(x,\sqrt n)}e^{-C_ld(x,y)^2/n}
\leq p_n(x,y)
\leq\frac{c_r}{V(x,\sqrt n)}e^{-C_rd(x,y)^2/n}
\end{equation}
for all $x,y\in V$ and all integers
$n\geq\max\{1,d(x,y)\}$.
\end{enumerate}
\end{theorem}

Condition $(\Delta)$ requires a loop at every vertex, so Theorem
\ref{tmp} cannot be applied directly even to a simple lattice.  A
classical way to create loops is to replace the transition operator by
$P^2$.  The corresponding edge weight is
\[
\sum_{z\in V}\frac{\mu_{xz}\mu_{zy}}{\mu(z)}.
\]
Although this construction is useful, preservation of $\mathrm{(PI)}$ is not immediate because all original edges may be lost.  To
avoid this inconvenience, and following the idea of \cite{CG2,De}, we
retain half of the original weight and define
\begin{equation}\label{lem3-1}
\widehat\mu_{xy}
=\frac12\mu_{xy}
 +\frac12\sum_{z\in V}\frac{\mu_{xz}\mu_{zy}}{\mu(z)}.
\end{equation}
Let $(V,\widehat\mu)$ be the graph in which $x$ and $y$ are adjacent
exactly when $\widehat\mu_{xy}>0$.  Denote its transition operator,
distance, volume, heat kernel, and Green function by
$\widehat P$, $\widehat d$, $\widehat V$, $\widehat p_n$, and
$\widehat g$, respectively.  Summing \eqref{lem3-1} over $y$ gives
\begin{equation}\label{equi-v}
\widehat\mu(x)=\mu(x),
\end{equation}
and consequently
\begin{equation}\label{phat}
\widehat P=\frac12(P+P^2).
\end{equation}

We now give the relationship between the Green functions corresponding
to the two weights $\mu$ and $\widehat\mu$.

\begin{lemma}\label{gg-rel}\rm
For all $x,y\in V$,
\begin{equation}\label{gghat}
\frac12g(x,y)\leq\widehat g(x,y)\leq g(x,y).
\end{equation}
\end{lemma}

\begin{proof}
We give the coefficient computation in detail.  First, by
\eqref{lem3-1} and Tonelli's theorem,
\begin{align*}
\widehat\mu(x)
&=\sum_{y\in V}\widehat\mu_{xy}\\
&=\frac12\sum_{y\in V}\mu_{xy}
 +\frac12\sum_{z\in V}\frac{\mu_{xz}}{\mu(z)}
       \sum_{y\in V}\mu_{zy}\\
&=\frac12\mu(x)+\frac12\sum_{z\in V}\mu_{xz}
=\mu(x).
\end{align*}
Thus \eqref{phat} holds entrywise:
\[
\widehat P(x,y)=\frac12P(x,y)+\frac12P^2(x,y).
\]
Since $P$ commutes with $P^2$, the binomial formula gives, for every
$n\geq0$,
\begin{equation}\label{eq:phat-binomial}
\widehat P^n
=2^{-n}(P+P^2)^n
=2^{-n}\sum_{m=0}^n\binom{n}{m}P^{n+m}.
\end{equation}
All coefficients are nonnegative, so by Tonelli's theorem,
\begin{equation}\label{eq:green-coefficients}
\sum_{n=0}^\infty\widehat P^n
=\sum_{k=0}^\infty c_kP^k,
\end{equation}
where
\[
c_k=\sum_{\substack{n\geq0,\ 0\leq m\leq n\\n+m=k}}
2^{-n}\binom{n}{m}.
\]
To identify $c_k$, consider its generating function.  For $0\leq t<1$, 
\begin{align*}
\sum_{k=0}^\infty c_kt^k
&=\sum_{n=0}^\infty2^{-n}
  \sum_{m=0}^n\binom{n}{m}t^{n+m}\\
&=\sum_{n=0}^\infty2^{-n}
t^{n}(1+t)^n\\
&=\sum_{n=0}^\infty\left(\frac{t+t^2}{2}\right)^n\\
&=\frac{1}{1-(t+t^2)/2}\\
&=\frac{2}{3}\frac{1}{1-t}
 +\frac{1}{3}\frac{1}{1+t/2}.
\end{align*}

Since
\[\frac1{1-t}=\sum_{k=0}^\infty t^k,
\qquad
\frac1{1+t/2}
=\sum_{k=0}^\infty\left(-\frac12\right)^k t^k,\]
comparing coefficients yields
\begin{equation}\label{eq:ck-explicit}
c_k=\frac23+\frac13\left(-\frac12\right)^k,
\qquad k\geq0.
\end{equation}
In particular, $c_0=1$, $c_1=1/2$, and
\begin{equation}\label{eq:ck-bounds}
\frac12\leq c_k\leq1
\qquad\text{for every }k\geq0.
\end{equation}

Taking the $(x,y)$ entry in \eqref{eq:green-coefficients} and using
$\widehat\mu(y)=\mu(y)$, we obtain
\[
\widehat g(x,y)
=\frac{1}{\mu(y)}\sum_{k=0}^\infty c_kP^k(x,y).
\]
On the other hand,
\[
g(x,y)=\frac{1}{\mu(y)}\sum_{k=0}^\infty P^k(x,y).
\]
The termwise bounds \eqref{eq:ck-bounds} now give
\[
\frac12g(x,y)\leq\widehat g(x,y)\leq g(x,y).
\]
The argument is valid in the extended interval $[0,+\infty]$; in the
transient case both Green functions are finite.  This completes the
proof.
\end{proof}

\begin{proposition}\label{Prop-equiv}\rm
The auxiliary graph has the following properties:
\begin{enumerate}
\item[(1)] $(P_0)$ on $(V,\mu)$ implies $(\Delta)$ on
$(V,\widehat\mu)$;
\item[(2)] $\mathrm{(VD)}$ on $(V,\mu)$ implies $\mathrm{(VD)}$ on
$(V,\widehat\mu)$;
\item[(3)] $\mathrm{(PI)}$ on $(V,\mu)$ implies $\mathrm{(PI)}$ on
$(V,\widehat\mu)$.
\end{enumerate}
\end{proposition}

\begin{proof}
Assume that $(V,\mu)$ satisfies $(P_0)$ with constant $\alpha>0$.
If $x\neq y$ and $d(x,y)=1$, then
\[
\widehat P(x,y)\geq\frac12P(x,y)\geq\frac{\alpha}{2}.
\]
If $d(x,y)=2$, choose $z$ on a path $x\sim z\sim y$.  Then
\[
\widehat P(x,y)\geq\frac12P^2(x,y)
\geq\frac12P(x,z)P(z,y)
\geq\frac{\alpha^2}{2}.
\]
Finally, for every $x$, choose a neighbor $z$ of $x$.  Reversibility and
$(P_0)$ give
\[
\widehat P(x,x)
\geq\frac12P(x,z)P(z,x)
\geq\frac{\alpha^2}{2}.
\]
Thus every vertex has a loop and every positive transition probability
of the auxiliary graph is bounded below by a constant depending only on
$\alpha$.  This proves (1).

A new edge joins vertices whose original distance is at most two.
Therefore a path of length $n$ in the new graph gives an original path
of length at most $2n$, and
\[
\widehat B(x,n)\subset B(x,2n).
\]
Conversely, an original geodesic of length at most $2n$ can be grouped
into consecutive pieces of length two, with possibly one final piece of
length one.  Every such piece is an edge of the auxiliary graph.  Hence
\[
B(x,2n)\subset\widehat B(x,n).
\]
We have proved the exact identity
\begin{equation}\label{eq:ball-identity}
\widehat B(x,n)=B(x,2n),
\qquad n\in\mathbb N.
\end{equation}
Together with $\widehat\mu=\mu$, this identity shows that volume
doubling is preserved, proving (2).

It remains to prove (3).  It is enough to consider integer radii, since
graph balls are constant between consecutive integers.  By
\eqref{eq:ball-identity},
\[
\widehat B(x_0,r)=B(x_0,2r),
\qquad
\widehat B(x_0,2r)=B(x_0,4r),
\]
and the averages with respect to $\mu$ and $\widehat\mu$ coincide.
Applying the original Poincar\'e inequality on the ball $B(x_0,2r)$
gives
\begin{align*}
\sum_{x\in\widehat B(x_0,r)}
 |f(x)-f_{\widehat B(x_0,r)}|^2\widehat\mu(x)
&\lesssim r^2
\sum_{x,y\in B(x_0,4r)}
 \mu_{xy}|f(y)-f(x)|^2.
\end{align*}
Since \eqref{lem3-1} implies
$\mu_{xy}\leq2\widehat\mu_{xy}$ and
$B(x_0,4r)=\widehat B(x_0,2r)$, the right-hand side is bounded by
\[
Cr^2\sum_{x,y\in\widehat B(x_0,2r)}
\widehat\mu_{xy}|f(y)-f(x)|^2.
\]
This is precisely the Poincar\'e inequality on the auxiliary graph.
\end{proof}

We next isolate the elementary summation argument that converts a
Gaussian heat-kernel estimate into a Green function estimate.

\begin{lemma}\label{lem:gaussian-green-sum}\rm
Let $W:[1,\infty)\to(0,\infty)$ be nondecreasing and doubling.  For
$c>0$ and every integer $\rho\geq1$,
\[
\sum_{n=\rho}^\infty
\frac{1}{W(\sqrt n)}e^{-c\rho^2/n}
\simeq
\sum_{r=\rho}^\infty\frac{r}{W(r)},
\]
where the constants depend only on $c$ and the doubling constant.
\end{lemma}

\begin{proof}
We first consider $n\geq\rho^2$.  For every integer $r\geq\rho$, group
the terms with
\[
r^2\leq n<(r+1)^2.
\]
There are $2r+1\simeq r$ such integers.  On this block,
$\sqrt n\simeq r$, so monotonicity and doubling give
$W(\sqrt n)\simeq W(r)$.  Moreover,
$e^{-c}\leq e^{-c\rho^2/n}\leq1$.  Consequently,
\begin{equation}\label{eq:long-time-sum}
\sum_{n=\rho^2}^\infty
\frac{1}{W(\sqrt n)}e^{-c\rho^2/n}
\simeq\sum_{r=\rho}^\infty\frac{r}{W(r)}.
\end{equation}

It remains to control the shorter range $\rho\leq n<\rho^2$.  Divide it
into dyadic blocks
\[
I_j=\{n:\ 2^{-j-1}\rho^2\leq n<2^{-j}\rho^2\},
\qquad j\geq0,
\]
retaining only the nonempty blocks that meet $[\rho,\rho^2)$.  On
$I_j$, one has $\sqrt n\simeq2^{-j/2}\rho$ and
$e^{-c\rho^2/n}\leq e^{-c2^j}$.  Iterating the doubling condition gives
\[
\frac{1}{W(\sqrt n)}\lesssim\frac{C_0^j}{W(\rho)}
\]
for a constant $C_0$ depending only on the doubling constant.  Since
$|I_j|\lesssim2^{-j}\rho^2$,
\begin{align*}
\sum_{n=\rho}^{\rho^2-1}
\frac{1}{W(\sqrt n)}e^{-c\rho^2/n}
&\lesssim\frac{\rho^2}{W(\rho)}
\sum_{j\geq0}2^{-j}C_0^je^{-c2^j}\\
&\lesssim\frac{\rho^2}{W(\rho)}.
\end{align*}
The last series converges because the exponential factor decays faster
than any geometric sequence.  Finally, monotonicity and doubling imply
\[
\sum_{r=\rho}^{2\rho}\frac{r}{W(r)}
\gtrsim\frac{\rho^2}{W(2\rho)}
\gtrsim\frac{\rho^2}{W(\rho)}.
\]
Thus the short-time contribution is bounded by the right-hand side of
\eqref{eq:long-time-sum}.  Together with \eqref{eq:long-time-sum}, this
proves the desired two-sided comparison.
\end{proof}

With the auxiliary-graph properties and the summation lemma in place, we
can now prove Theorem \ref{GF-est}.

\begin{proof}[Proof of Theorem \ref{GF-est}]
By Proposition \ref{Prop-equiv}, the auxiliary graph
$(V,\widehat\mu)$ satisfies $(\Delta)$, $\mathrm{(VD)}$, and
$\mathrm{(PI)}$.  Theorem \ref{tmp} therefore gives, with
$\rho=\widehat d(x,y)\geq1$,
\[
\widehat g(x,y)
=\sum_{k=0}^\infty\widehat p_k(x,y)
\simeq\sum_{k=\rho}^\infty
\frac{1}{\widehat V(x,\sqrt k)}
 e^{-c\rho^2/k}.
\]
Applying Lemma \ref{lem:gaussian-green-sum} with
$W(r)=\widehat V(x,r)$ yields
\begin{equation}\label{eq:ghat-volume-tail}
\widehat g(x,y)
\simeq\sum_{n=\widehat d(x,y)}^\infty
\frac{n}{\widehat V(x,n)}.
\end{equation}
For $x=y$, the same conclusion follows by taking $\rho=1$ and absorbing
the time-zero term into the constants; condition $\mathrm{(VD)}$ makes
$\widehat V(x,1)$ comparable with $\widehat\mu(x)$.

By the construction of the auxiliary graph,
\[
\widehat d(x,y)\leq d(x,y)\leq2\widehat d(x,y),
\qquad
\widehat V(x,n)=V(x,2n).
\]
Volume doubling gives $V(x,2n)\simeq V(x,n)$.  It also shows that
changing the lower limit of the tail by a factor at most two changes the
sum only by a multiplicative constant.  Hence
\[
\sum_{n=\widehat d(x,y)}^\infty
\frac{n}{\widehat V(x,n)}
\simeq
\sum_{n=d(x,y)}^\infty\frac{n}{V(x,n)}.
\]
Combining this comparison with \eqref{eq:ghat-volume-tail} and Lemma
\ref{gg-rel} proves \eqref{green-est-intro}.
\end{proof}

Using the Green function estimate, we are now ready to prove Theorem \ref{tvcon2}.
\begin{proof}[Proof of Theorem \ref{tvcon2}]
Decompose $V$ into the disjoint spheres $S(o,n)$.  By symmetry of the
Green function,
\[
g_q(o,o)
=\sum_{z\in V}g(z,o)^q\mu(z)
=\sum_{n=0}^\infty\sum_{z\in S(o,n)}g(o,z)^q\mu(z).
\]
For every $z\in S(o,n)$, one has $d(o,z)=n$.  Theorem \ref{GF-est}, with
the volume centered at $o$, therefore gives uniformly on the sphere
\[
g(o,z)
\simeq\sum_{m=n}^\infty\frac{m}{V(o,m)}.
\]
It follows that
\begin{align*}
g_q(o,o)
&\simeq\sum_{n=0}^\infty
\left(\sum_{m=n}^\infty\frac{m}{V(o,m)}\right)^q
\sum_{z\in S(o,n)}\mu(z)\\
&=\sum_{n=0}^\infty
\left(\sum_{m=n}^\infty\frac{m}{V(o,m)}\right)^q
V(S(o,n)).
\end{align*}
Thus the series in \eqref{bdg} diverges exactly when
$g_q(o,o)=\infty$.  The equivalence now follows from Theorem
\ref{tm1}.
\end{proof}

\begin{theorem}\label{vcond}\rm
Let $1<p<\infty$, and assume that $(V,\mu)$ satisfies
$\mathrm{(VD)}$, $\mathrm{(PI)}$, and $(P_0)$.  If, for some $o\in V$,
\begin{equation}\label{eq:vcond}
\sum_{n=1}^\infty n
\left(\sum_{m=n}^\infty\frac{m}{V(o,m)}\right)^{1/(p-1)}
=\infty,
\end{equation}
then $(V,\mu)$ is $L^p$-parabolic.
\end{theorem}

\begin{proof}
Set
\[
a_n=\sum_{m=n}^\infty\frac{m}{V(o,m)},
\qquad n\geq0.
\]
If $a_n=\infty$ for some $n$, then the $n$-th summand in
\eqref{bdg} is infinite because every sphere $S(o,n)$ is nonempty on an
infinite, connected, locally finite graph; Theorem \ref{tvcon2} gives
the conclusion.  We may therefore
assume that every $a_n$ is finite.  The sequence is decreasing and
\begin{equation}\label{eq:an-difference}
a_n-a_{n+1}=\frac{n}{V(o,n)}.
\end{equation}

We shall also use $a_n\simeq a_{n+1}$ for $n\geq1$.  To see this, volume
doubling gives
\begin{align*}
a_{n+1}
&\geq\sum_{m=n+1}^{2n}\frac{m}{V(o,m)}\\
&\gtrsim\frac{n^2}{V(o,2n)}
\gtrsim\frac{n^2}{V(o,n)}.
\end{align*}
Together with \eqref{eq:an-difference}, this implies
$a_n\leq(1+C/n)a_{n+1}\lesssim a_{n+1}$; the reverse inequality is
immediate from monotonicity.

Write $V_n=V(o,n)$ and set $V_{-1}=0$.  For every $N\geq1$, discrete
summation by parts gives
\begin{align}
\sum_{n=0}^Na_n^q(V_n-V_{n-1})
&=a_N^qV_N
 +\sum_{n=0}^{N-1}(a_n^q-a_{n+1}^q)V_n. \label{eq:abel-capacity}
\end{align}
Since $q>1$ and $a_n\geq a_{n+1}$, the mean value theorem yields
\[
a_n^q-a_{n+1}^q
\geq q a_{n+1}^{q-1}(a_n-a_{n+1}).
\]
Discarding the nonnegative boundary term in
\eqref{eq:abel-capacity}, using \eqref{eq:an-difference}, and then using
$a_n\simeq a_{n+1}$, we obtain
\begin{align*}
\sum_{n=0}^Na_n^q(V_n-V_{n-1})
&\geq q\sum_{n=1}^{N-1}
 a_{n+1}^{q-1}\frac{n}{V_n}V_n\\
&=q\sum_{n=1}^{N-1}n a_{n+1}^{q-1}\\
&\gtrsim\sum_{n=1}^{N-1}n a_n^{q-1}.
\end{align*}
Because $q-1=1/(p-1)$, condition \eqref{eq:vcond} forces the left-hand
side to tend to $+\infty$ as $N\to\infty$.  This is precisely the
divergence condition \eqref{bdg}.  Theorem \ref{tvcon2} completes the
proof.
\end{proof}

\begin{corollary}\label{co2}\rm
Let $1<p<\infty$, and assume that $(V,\mu)$ satisfies
$\mathrm{(VD)}$, $\mathrm{(PI)}$, and $(P_0)$.  If, for some $o\in V$,
\[
V(o,r)\lesssim r^{2p}(\log r)^{p-1}
\qquad\text{for all sufficiently large }r,
\]
then $(V,\mu)$ is $L^p$-parabolic.
\end{corollary}

\begin{proof}
Let $n$ be sufficiently large.  Restricting the tail to the interval
$n\leq m\leq2n$ and using the assumed upper bound for the volume gives
\begin{align*}
\sum_{m=n}^\infty\frac{m}{V(o,m)}
&\geq\sum_{m=n}^{2n}\frac{m}{V(o,m)}\\
&\gtrsim\sum_{m=n}^{2n}
 m^{1-2p}(\log m)^{1-p}\\
&\gtrsim n^{2-2p}(\log n)^{1-p}.
\end{align*}
Raising this estimate to the power $1/(p-1)$ and multiplying by $n$
yields
\[
n\left(\sum_{m=n}^\infty\frac{m}{V(o,m)}\right)^{1/(p-1)}
\gtrsim\frac{1}{n\log n}.
\]
The harmonic series with logarithmic weight on the right diverges, so
condition \eqref{eq:vcond} holds.  The conclusion follows from Theorem
\ref{vcond}.
\end{proof}

\section{Examples}\label{examples}

We finish by applying the exact series criterion to Cayley graphs, where
the geometric hypotheses and the volume tail can be computed explicitly.
Let $G$
be a finitely generated group and let $S\subset G$ be a finite symmetric
generating set, meaning that $s\in S$ implies $s^{-1}\in S$.  The group
$G$ and the set $S$ determine a graph as follows: the vertex set is $G$,
and
\[
x\sim y\quad\Longleftrightarrow\quad x^{-1}y\in S.
\]
We use the normalized edge weight
\[
\mu_{xy}=
\begin{cases}
|S|^{-1},&x^{-1}y\in S,\\
0,&\text{otherwise}.
\end{cases}
\]
Then $\mu(x)=1$ for every $x\in G$.  If the identity element $e$ belongs
to $S$, every vertex has a loop; otherwise the graph has no loops.  By
left translation invariance,
\[
V(x,n)=V(e,n),
\qquad x\in G.
\]

\begin{definition}\rm
An infinite weighted graph $(V,\mu)$ has polynomial growth of degree
$D>0$ if, for some $o\in V$,
\begin{equation}\label{6-1}
V(o,r)\simeq r^D,
\qquad r\geq1.
\end{equation}
\end{definition}

The next proposition shows that this growth exponent is independent of
the chosen finite generating set.

\begin{proposition}\rm
Let $(G,\mu)$ and $(G,\mu')$ be Cayley graphs associated with finite
generating sets $S$ and $S'$, respectively.  If $(G,\mu)$ has
polynomial growth of degree $D$, then $(G,\mu')$ has polynomial growth
of the same degree.
\end{proposition}

\begin{proof}
Let $d$ and $d'$ be the word metrics corresponding to $S$ and $S'$.
Since $S'$ is finite, the number
\[
a=\max_{s'\in S'}d(e,s')
\]
is finite.  If $z\in B'(e,n)$, then $z$ can be written as a product of
at most $n$ elements of $S'$.  Replacing each of these elements by an
$S$-word of length at most $a$, we obtain
\[
d(e,z)\leq a d'(e,z)\leq an.
\]
Thus $B'(e,n)\subset B(e,an)$, and hence
\[
V'(e,n)\leq V(e,an)\lesssim n^D.
\]
Interchanging $S$ and $S'$, there is a constant $b>0$ such that
$B(e,n)\subset B'(e,bn)$.  Therefore
\[
V'(e,bn)\geq V(e,n)\gtrsim n^D.
\]
After replacing $n$ by a comparable integer, this gives
$V'(e,r)\gtrsim r^D$.  Combining the two estimates proves the claim.
\end{proof}

This invariance makes the following group-level definition unambiguous.

\begin{definition}\rm
A finitely generated group has polynomial growth if one, equivalently
every, Cayley graph associated with a finite generating set has
polynomial growth.
\end{definition}

The next proposition verifies the Poincar\'e inequality needed in the
preceding section.

\begin{proposition}\label{pro6}\rm
Let $(G,\mu)$ be a Cayley graph associated with a finite generating set.
If $(G,\mu)$ satisfies $\mathrm{(VD)}$, then it also satisfies
$\mathrm{(PI)}$.
\end{proposition}

\begin{remark}\rm
This group-level Poincar\'e estimate follows from the standard path
argument used in \cite{CS}.  For completeness and convenience, we include
the proof.
\end{remark}

\begin{proof}
Fix $x_0\in G$, a positive integer $n$, and $f\in\ell(G)$.  Put
\[
f_B=\frac{1}{V(x_0,n)}
\sum_{x\in B(x_0,n)}f(x).
\]
By Jensen's inequality,
\begin{align}
\sum_{x\in B(x_0,n)}|f(x)-f_B|^2
&\leq\frac{1}{V(x_0,n)}
\sum_{x,y\in B(x_0,n)}|f(x)-f(y)|^2\nonumber\\
&=\frac{1}{V(x_0,n)}
\sum_{z\in B(e,2n)}\sum_{x\in A_z}|f(x)-f(xz)|^2, \label{ex-1}
\end{align}
where
\[
A_z=\{x\in G:\ x\in B(x_0,n),\ xz\in B(x_0,n)\}.
\]
The last identity follows by writing $y=xz$; for a pair
$x,y\in B(x_0,n)$, the element $z=x^{-1}y$ belongs to $B(e,2n)$.

Fix $z\in B(e,2n)$ and choose a geodesic word
\[
z=s_1\cdots s_k,
\qquad s_i\in S,
\qquad k=d(e,z)\leq2n.
\]
For $x\in A_z$, telescope along the path
$x,xs_1,\ldots,xs_1\cdots s_k=xz$.  The Cauchy--Schwarz inequality
gives
\begin{align}
|f(x)-f(xz)|^2
&\leq k\sum_{i=1}^k
|f(xs_1\cdots s_{i-1})-f(xs_1\cdots s_i)|^2\nonumber\\
&\leq2n\sum_{i=1}^k
|f(xs_1\cdots s_{i-1})-f(xs_1\cdots s_i)|^2. \label{ex-2}
\end{align}
Every intermediate vertex lies in $B(x_0,2n)$.  Indeed, its distance
from $x_0$ is at most both $n+i$ and $n+k-i$, and hence at most
$n+k/2\leq2n$.  For a fixed $i$, right multiplication by
$s_1\cdots s_{i-1}$ is injective.  Since every edge weight is
$|S|^{-1}$, summing over $x\in A_z$ gives
\begin{equation}\label{ex-3}
\sum_{x\in A_z}
|f(xs_1\cdots s_{i-1})-f(xs_1\cdots s_i)|^2
\leq |S|\sum_{u,v\in B(x_0,2n)}
\mu_{uv}|f(u)-f(v)|^2.
\end{equation}
There are at most $2n$ indices $i$.  Combining \eqref{ex-2} and
\eqref{ex-3}, we obtain
\begin{equation}\label{ex-4}
\sum_{x\in A_z}|f(x)-f(xz)|^2
\leq4n^2|S|\sum_{u,v\in B(x_0,2n)}
\mu_{uv}|f(u)-f(v)|^2.
\end{equation}

Substitution of \eqref{ex-4} into \eqref{ex-1} yields
\[
\sum_{x\in B(x_0,n)}|f(x)-f_B|^2
\leq4\frac{V(e,2n)}{V(x_0,n)}|S|n^2
\sum_{u,v\in B(x_0,2n)}\mu_{uv}|f(u)-f(v)|^2.
\]
By left translation invariance, $V(x_0,n)=V(e,n)$.  Moreover, volume
doubling gives
\[
\frac{V(e,2n)}{V(x_0,n)}
=\frac{V(e,2n)}{V(e,n)}\lesssim1.
\]
Therefore
\[
\sum_{x\in B(x_0,n)}|f(x)-f_B|^2
\lesssim n^2\sum_{u,v\in B(x_0,2n)}
\mu_{uv}|f(u)-f(v)|^2,
\]
which is the required Poincar\'e inequality.
\end{proof}

Combining the preceding Poincar\'e estimate with the exact criterion from
Section \ref{volume} gives the main application.

\begin{corollary}\label{co6}\rm
If a finitely generated group $G$ has polynomial growth
\eqref{6-1} of degree $D$, then every Cayley graph $(G,\mu)$ associated
with a finite generating set satisfies $\mathrm{(VD)}$ and
$\mathrm{(PI)}$.  Moreover, for every $1\leq p<\infty$,
\[
(G,\mu)\text{ is }L^p\text{-parabolic}
\quad\Longleftrightarrow\quad
p\geq\frac D2.
\]
\end{corollary}

\begin{proof}
The growth estimate \eqref{6-1} gives
\[
V(e,2r)\lesssim(2r)^D\lesssim V(e,r),
\]
so the volume doubling condition holds.  Proposition \ref{pro6} then
gives the Poincar\'e inequality.  Since the normalized edge weights are
$|S|^{-1}$, the Cayley graph also satisfies $(P_0)$.

Suppose first that $D\leq2$.  Then
\[
\sum_{n=1}^{\infty}\frac{n}{V(e,n)}
\gtrsim\sum_{n=1}^{\infty}n^{1-D}=\infty.
\]
Hence $(G,\mu)$ is parabolic by \eqref{votest}, and therefore it is
$L^p$-parabolic for every $p\geq1$ by Corollary
\ref{st-parabolic}.  Since $D/2\leq1$, this proves the assertion in this
case.

Assume now that $D>2$.  Theorem \ref{GF-est} and
$V(e,n)\simeq n^D$ show that the Green function is finite, because
\[
\sum_{n=1}^{\infty}\frac{n}{V(e,n)}
\simeq\sum_{n=1}^{\infty}n^{1-D}<\infty.
\]
Thus the graph is not $L^1$-parabolic.  Let $1<p<\infty$, let
$q=p/(p-1)$, and set
\[
A_n:=\sum_{m=n}^{\infty}\frac{m}{V(e,m)}.
\]
The polynomial volume growth gives
\begin{equation}\label{eq:cayley-tail}
A_n\simeq n^{2-D}.
\end{equation}
By Theorem \ref{tvcon2} and \eqref{eq:cayley-tail}, the graph is
$L^p$-parabolic if and only if
\begin{equation}\label{eq:cayley-stieltjes}
\sum_{n=1}^{\infty}n^{-\beta}
\bigl(V(e,n)-V(e,n-1)\bigr)=\infty,
\qquad \beta:=q(D-2).
\end{equation}

We claim that the series in \eqref{eq:cayley-stieltjes} diverges if and
only if $\beta\leq D$.  If $\beta>D$, decomposition into dyadic blocks
gives
\begin{align*}
\sum_{n=2^k}^{2^{k+1}-1}n^{-\beta}
\bigl(V(e,n)-V(e,n-1)\bigr)
&\leq 2^{-k\beta}V(e,2^{k+1})\\
&\lesssim 2^{k(D-\beta)},
\end{align*}
and the sum over $k$ converges.  Conversely, choose constants $c,C>0$
such that
\[
cr^D\leq V(e,r)\leq Cr^D
\]
for all large $r$, and then choose an integer $L\geq2$ so large that
$cL^D>C$.  On the disjoint blocks $L^k<n\leq L^{k+1}$, one has
\begin{align*}
&\sum_{n=L^k+1}^{L^{k+1}}n^{-\beta}
\bigl(V(e,n)-V(e,n-1)\bigr)\\
&\qquad\geq L^{-(k+1)\beta}
\bigl(V(e,L^{k+1})-V(e,L^k)\bigr)
\gtrsim L^{k(D-\beta)}.
\end{align*}
The sum of these block contributions diverges whenever $\beta\leq D$,
which proves the claim.

Finally,
\[
\beta\leq D
\quad\Longleftrightarrow\quad
\frac{p}{p-1}(D-2)\leq D
\quad\Longleftrightarrow\quad
p\geq\frac D2.
\]
This proves both directions, including the critical case $p=D/2$.
\end{proof}

\begin{remark}\rm
Let us emphasize that finitely generated groups of polynomial growth
form a large and well-studied class.  By Gromov's theorem \cite{Gro}, a
finitely generated group has polynomial growth if and only if it is
virtually nilpotent, that is, it contains a nilpotent subgroup of finite
index.
\end{remark}

We now specialize Corollary \ref{co6} to two standard groups.

\begin{example}\rm
For $\mathbb Z^d$, one has
\[
V(e,n)\simeq n^d,
\qquad
V(S(e,n))\simeq n^{d-1}.
\]

If $d\leq2$, then $\mathbb Z^d$ is recurrent and hence is
$L^p$-parabolic for every $p\geq1$.  Let $d>2$ and
$q=p/(p-1)$.  Then
\[
\sum_{m=n}^{\infty}\frac{m}{V(e,m)}\simeq n^{2-d},
\]
so the series in \eqref{bdg} is comparable to
\[
\sum_{n=1}^{\infty}n^{d-1-q(d-2)}.
\]
It diverges if and only if $q(d-2)\leq d$, which is equivalent to
$p\geq d/2$.  Hence
$\mathbb Z^d$ is $L^p$-parabolic if and only if $p\geq d/2$.
\end{example}

\begin{example}\rm
The discrete Heisenberg group
\[
\left\{
\begin{pmatrix}
1&b&c\\
0&1&a\\
0&0&1
\end{pmatrix}
:\ a,b,c\in\mathbb Z
\right\}
\]
has polynomial growth of degree $4$.  Therefore its Cayley graphs are
$L^p$-parabolic if and only if $p\geq2$ by Corollary \ref{co6}.
\end{example}

These examples show that the abstract Green kernel criterion recovers the
critical exponent from the degree of polynomial volume growth, including the
borderline case.

\textbf{Acknowledgements.}
The authors thank Prof. Grigor'yan of Bielefeld University for many valuable suggestions and discussions.

\end{document}